\documentclass[a4paper]{amsart}
\usepackage{amsmath,amsfonts,latexsym,amssymb,amsthm, verbatim}
\usepackage{booktabs}
\usepackage{rotating}
\usepackage{url}

\newtheorem{tm}{Theorem}
\newtheorem{proposition}[tm]{Proposition}
\newtheorem{lemma}[tm]{Lemma}
\newtheorem{corollary}[tm]{Corollary}

\theoremstyle{definition}
\newtheorem*{definition}{Definition}
\newtheorem{example}{Example}

\theoremstyle{remark}
\newtheorem*{acknowledgments}{Acknowledgments}
\newtheorem*{remark}{Remark}
\newtheorem*{remarks}{Remarks}

\newcommand{\R}{\mathbb{R}}
\newcommand{\Q}{\mathbb{Q}}
\newcommand{\Qbar}{\overline{\Q}}
\newcommand{\C}{\mathbb{C}}
\newcommand{\Z}{\mathbb{Z}}

\newcommand{\F}{\mathbb{F}}
\newcommand{\PP}{\mathbb{P}}

\newcommand{\J}{\mathrm{J}}
\newcommand{\X}{\mathrm{X}}
\newcommand{\Y}{\mathrm{Y}}

\newcommand{\tors}{\mathrm{tors}}

\DeclareMathOperator{\Aut}{Aut}
\DeclareMathOperator{\End}{End}
\DeclareMathOperator{\Gal}{Gal}

\DeclareMathOperator{\Hom}{Hom}
\DeclareMathOperator{\Res}{Res}
\DeclareMathOperator{\Spec}{Spec}
\DeclareMathOperator{\Sym}{Sym}

\def\diam#1{\langle#1\rangle}
\def\isom{\buildrel\sim\over\longrightarrow}
\def\lsup#1#2{{}^{#1}\!{#2}}

\begin{document}

\title{Ranks of elliptic curves with prescribed torsion over number fields}

\author{Johan Bosman}
\address{Mathematics Institute\\ University of Warwick\\ Coventry CV4 7AL\\ United Kingdom}
\email{J.Bosman@warwick.ac.uk}
\thanks{J.B.\ was supported by Marie Curie FP7 grant 252058.}
\author{Peter Bruin}
\address{Institut f\"ur Mathematik\\ Universit\"at Z\"urich\\ Winterthurerstrasse 190\\ 8057 Z\"urich\\ Switzerland}
\email{peter.bruin@math.uzh.ch}
\thanks{P.B. was supported by the program ``Points entiers et points rationnels'' (Agence nationale de la recherche, France), I.H.\'E.S., and Swiss National Science Foundation grant 124737.}
\author{Andrej Dujella}
\address{Department of Mathematics\\ University of Zagreb\\ Bijeni\v cka cesta 30\\ 10000 Zagreb\\ Croatia}
\thanks{A.D. and F.N. were supported by the Ministry of Science, Education, and Sports, Republic of Croatia, grant 037-0372781-2821.}
\email{duje@math.hr}
\author{Filip Najman}
\address{Department of Mathematics, University of Zagreb, Bijeni\v cka cesta 30, 10000 Zagreb, Croatia}
\email{fnajman@math.hr}
\thanks{F.N. was also supported by the National Foundation for Science, Higher Education, and Technological Development of the Republic of Croatia}

\date{}


\begin{abstract}
We study the structure of the Mordell--Weil group of elliptic curves over number fields of degree 2, 3, and~4.  We show that if $T$ is a group, then either the class of all elliptic curves over quadratic fields with torsion subgroup~$T$ is empty, or it contains curves of rank~0 as well as curves of positive rank.  We prove a similar but slightly weaker result for cubic and quartic fields.  On the other hand, we find a group~$T$ and a quartic field~$K$ such that among the elliptic curves over~$K$ with torsion subgroup~$T$, there are curves of positive rank, but none of rank~0.  We find examples of elliptic curves with positive rank and given torsion in many previously unknown cases.  We also prove that all elliptic curves over quadratic fields with a point of order 13 or 18 and all elliptic curves over quartic fields with a point of order 22 are isogenous to one of their Galois conjugates and, by a phenomenon that we call \emph{false complex multiplication}, have even rank.  Finally, we discuss connections with elliptic curves over finite fields and applications to integer factorization.
\end{abstract}

\maketitle

\section{Introduction}

For an elliptic curve $E$ over a number field $K$, the Mordell--Weil theorem states that the Abelian group $E(K)$ of $K$-rational points on $E$ is finitely generated.
The group $E(K)$ is isomorphic to $T\oplus\Z^r$, where $T$ is the torsion subgroup and $r$ is a non-negative integer called the \emph{rank} of the elliptic curve.

The aim of this paper is to study the interplay of the rank and torsion group of elliptic curves over quadratic, cubic, and quartic number fields. More explicitly, we will be interested in the following question: given a torsion group $T$ and an elliptic curve $E$ over a number field $K$ of degree 2, 3, or 4 such that $E(K)_\tors \simeq T$, what can we say about the rank of~$E(K)$?


A natural starting point is to wonder what can be said over~$\Q$. Let $T$ be a possible torsion group of an elliptic curve over $\Q$, i.e.\ it is one of the 15 groups from Mazur's torsion theorem \cite[Theorem~8]{maz}. What can be said about the rank? The short answer is: nothing. For all we know, it might be 0, it might be positive, it might be even, or it might be odd. As we will later show, this is in stark contrast with what happens over number fields. Note that already in \cite{kn} torsion groups $T$ and number fields $K$ are given such that every elliptic curve over $K$ with torsion $T$ has rank~0.

We start by examining a very basic question: how many points can an elliptic curve have over a quadratic, cubic, or quartic field and what group structure can these points have? As there exist elliptic curves with infinitely many points over~$\Q$, it is clear that an elliptic curve can possibly have infinitely many points over any number field. So, this question is closely related to determining for which finite groups $T$ there exists an elliptic curve over a field of degree 2, 3, or~4 with torsion~$T$ and rank~0.

Over $\Q$, the answer to this question is known by Mazur's torsion theorem \cite[Theorem~8]{maz}; it is an easy exercise to find an elliptic curve with fixed torsion and rank 0 over~$\Q$. An elliptic curve over $\Q$ has 1, \dots, 10, 12, 16, or infinitely many rational points.

We prove a similar result over all the quadratic fields using a theorem of Kamienny, Kenku, and Momose \cite[Theorem~3.1]{Kam1}, \cite[Theorem~(0.1)]{km}, which tells us which groups appear as torsion of elliptic curves over quadratic fields.
We find, for each of these torsion groups, an elliptic curve over some quadratic field having that particular torsion group and rank~$0$.

Note that if one fixes a number field and not just the degree, the situation is a bit different.  Mazur and Rubin \cite[Theorem~1.1]{mr} have proved that for each number field $K$ there exists an elliptic curve over $K$ such that the rank of $E(K)$ is~0. A natural question is whether the following generalization is true: if $T$ is a group that appears as a torsion group of an elliptic curve over $K$, does there exist some elliptic curve with torsion $T$ and rank~0? We prove that the answer is no, by giving an example of a quartic field $K$ such that every elliptic curve with torsion $\Z/15 \Z$ over~$K$ has positive rank.

It is known which groups appear \emph{infinitely often} as a torsion group of an elliptic curve over a cubic field \cite[Theorem~3.4]{jks}.  For each group $T$ from this list, we find a curve over a cubic field with rank~0 and torsion group~$T$.  We do the same for quartic fields, using the analogous result \cite[Theorem~3.6]{jkp}.  It is not known if any other groups can appear as torsion groups of elliptic curves over cubic or quartic fields.


Next, we examine elliptic curves with given torsion and positive rank over fields of degree 2, 3, and~4. This problem has been extensively studied over $\Q$; see for example \cite{du} for a list of references and rank records with given torsion. Over quadratic fields Rabarison \cite[Section~4]{rab} found examples of elliptic curves with given torsion and positive rank for all except possibly 4 torsion groups. We find examples of elliptic curves with positive rank unconditionally for all these groups.

We do the same for cubic and quartic fields and find many instances of elliptic curves with previously unknown Mordell--Weil groups.

When searching for elliptic curves with given torsion and positive rank, we noticed that all constructed elliptic curves with 13-torsion and 18-torsion over quadratic fields and all constructed elliptic curves with 22-torsion over quartic fields appeared to have even rank.
In Section~\ref{sec:false-cm} we prove that this is not a coincidence and that all elliptic curves with 13-torsion and 18-torsion over quadratic fields  and all elliptic curves with 22-torsion over quartic fields have even rank indeed. The explanation involves $\Q$-curves and $K$-curves, where $K$ is a quadratic field, and a phenomenon that we call \emph{false complex multiplication}.

In Section \ref{sec:aff} we examine the connections and applications of our results to elliptic curves over finite fields.
In the elliptic curve factoring method \cite{len}, one looks for elliptic curves whose number of points over~$\F_p$ for $p$ prime is likely to be smooth, i.e.\ divisible only by small primes.
It is now a classical method (see \cite{am} and~\cite{mon}) to use elliptic curves $E$ with large rational torsion, as the torsion injects into $E(\F_p)$ for all primes $p\geq 3$ of good reduction.
This in turn makes the order of $E(\F_p)$ more likely to be smooth.

One can get more information about the heuristic probability of an elliptic curve to have smooth order if the torsion of the curve is examined over both the rationals and number fields of small degree.
We give an explicit example of two curves such that the one with smaller rational torsion has smooth order more often and discuss the implications of this for choosing elliptic curves for factoring.

We used Magma \cite{mag}, Pari \cite{Pari}, and Sage \cite{sag} for most of our computations.

\begin{acknowledgments}
We are greatly indebted to Hendrik~W.~Lenstra, Jr.\ for ideas that motivated much of Section 4, including the terminology ``false complex multiplication''. We thank Thomas Preu for finding points of infinite order on the elliptic curve with torsion $\Z/18\Z$ over a quadratic field mentioned in Theorem~\ref{t6}.
\end{acknowledgments}

\section{Curves with prescribed torsion and rank 0}

We first examine elliptic curves over quadratic fields. By a theorem of Kamienny \cite[Theorem~3.1]{Kam1}, Kenku, and Momose \cite[Theorem~(0.1)]{km}, the following 26 groups can appear as a torsion group of an elliptic curve over a quadratic field:
\begin{equation}
\vcenter{\openup\jot\halign{$\hfil#\hfil$\cr
\Z/m \Z \text{ for } 1 \leq m\leq 18,\ m\neq 17,\cr
\Z/2 \Z \oplus \Z/2m \Z \text{ for } 1 \leq m\leq 6,\cr
\Z/3 \Z \oplus \Z/3m \Z \text{ for }  m=1,2,\cr
\Z/4 \Z \oplus \Z/4 \Z.\cr}}
\label{eq1}
\end{equation}

We used the \textsf{RankBound()} function in Magma and the program \textsf{mwrank}~\cite{mwr} for rank computations. Over quadratic fields Magma can easily prove that the rank of the curves we are going to list in Theorem~\ref{t1} is~0, so one does not have to employ any tricks to check the rank that will be needed in higher degree cases.

Note that it is easy to find examples with a torsion group that appears over $\Q$ and rank~$0$. 
By standard conjectures (see \cite{gol}), half of all elliptic curves should have rank~0.
One first finds a curve $E/\Q$ with given torsion and rank~0.  
Then one finds a squarefree $d\in\Z$ such that the quadratic twist $E^{(d)}$ has rank~0. 
The curve $E_{\Q(\sqrt{d})}$ has rank $0$ as well, and its torsion is likely the same as that of $E_\Q$.
We still give explicit examples, as we feel that this is in the spirit of this paper. For the rest of the groups, one can find infinite families of elliptic curves with given torsion in \cite{jk4} and \cite[Section~4]{rab} and then search for rank~0 examples.

\begin{tm}
\label{t1}
For each group $T$ from (\ref{eq1}), there exists an elliptic curve over a quadratic field with torsion $T$ and rank 0.
\end{tm}
\begin{proof}
In Table~\ref{table:rank0quadratic} we give explicit examples of elliptic curves with rank~0.
For all the curves listed, one checks that the rank is indeed 0 using a 2-descent.
\end{proof}

It is known which torsion groups appear infinitely often over cubic fields \cite[Theorem~3.4]{jks}:
\begin{equation}
\vcenter{\openup\jot\halign{$\hfil#\hfil$\cr
\Z/m \Z \text{ for } 1 \leq m\leq 20,\ m\neq 17,19,\cr
\Z/2 \Z \oplus \Z/2m \Z \text{ for } 1 \leq m\leq 7.\cr}}
\label{eq2}
\end{equation}
There is no known example of an elliptic curve over a cubic field with a torsion subgroup that is not in the list $(\ref{eq2})$, although it has not been proved that there are no such curves.

As the computation of rank bounds becomes more time-consuming in the cubic case, we first compute the parity (assuming the Birch--Swinnerton-Dyer conjecture) of the rank of the elliptic curve using the \textsf{RootNumber()} function in Magma and eliminate the curves with odd rank.

\begin{tm}
\label{t2}
For each group $T$ from \eqref{eq2}, there exists an elliptic curve over a cubic field with torsion $T$ and rank 0.
\end{tm}
\begin{proof}
We give explicit examples in Table~\ref{table:rank0cubic}; again, one can easily check that each curve has rank~0. We obtained our curves from~\cite{jk3}.
\end{proof}

As with cubic fields, it is known which torsion groups appear infinitely often over quartic fields \cite[Theorem~3.6]{jkp}:
\begin{equation}
\vcenter{\openup\jot\halign{$\hfil#\hfil$\cr
\Z/m \Z \text{ for } 1 \leq m\leq 24,\ m\neq 19,23,\cr
\Z/2 \Z \oplus \Z/2m \Z \text{ for } 1 \leq m\leq 9,\cr
\Z/3 \Z \oplus \Z/3m \Z \text{ for } 1 \leq m\leq 3,\cr
\Z/4 \Z \oplus \Z/4m \Z \text{ for } 1 \leq m\leq 2,\cr
\Z/5 \Z \oplus \Z/5 \Z,\cr
\Z/6 \Z \oplus \Z/6 \Z.\cr}}
\label{eq3}\end{equation}
Again it is unknown whether there exists an elliptic curve over a quartic field with a torsion subgroup that is not in the list~$(\ref{eq3})$.

The computation of upper bounds on the rank of elliptic curves over quartic fields becomes much harder if it is done directly over the quartic field. However, we will use the fact that if an elliptic curve $E$ is defined over a number field $K$ and $L$ is an extension of $K$ of degree~2, i.e.\ $L=K(\sqrt d),\ d\in K$ and $d$ is not a square, then the rank of $E(L)$ is the sum of the rank of $E(K)$ and the rank of~$E^{(d)}(K)$. This can reduce the computation of the rank of an elliptic curve over a quartic field down to the computation of two ranks over quadratic fields or four computations over~$\Q$. Again, we always compute the conjectural parity of the rank of the elliptic curve before the actual computation of the rank.

\begin{tm}
\label{t3}
For each group $T$ from (\ref{eq3}) there exists an elliptic curve over a quartic field with torsion $T$ and rank 0.
\end{tm}
\begin{proof}
We give explicit examples of such curves in Table~\ref{table:rank0quartic}. We used curves defined over smaller fields wherever possible. Curves that could not be obtained in such a way were obtained from~\cite{jk4}.

Unlike in Tables \ref{table:rank0quadratic} and~\ref{table:rank0cubic}, here the computation of the rank is usually very hard. We solve one example, $(2,16)$, in detail to give a flavor of how our curves were obtained.

We start our search by looking for elliptic curves $E$ over the rationals with torsion $\Z/ 2\Z \oplus \Z/ 8\Z$ and rank~0, where $P$ is a point of order $8$ and $Q$ is of order $2$ such that $P$ and $Q$ generate the complete torsion. Next, we search for a point $R\in E(\Qbar)$ with~$2R=P$. There will be 4 such choices for~$R$. We can obtain the field of definition of the points $R$ by factoring the 16th division polynomial of our elliptic curve. There will be five factors of degree $4$: one for each choice of $R$, and a fifth one that will generate the field of definition of the point $Q_1$ satisfying $2Q_1=Q$ (over the field of definition of $Q_1$, our elliptic curve will have torsion isomorphic to $\Z/ 4\Z \oplus \Z/ 8\Z$). Over each field $K$ we check the conjectural parity of our starting elliptic curve over~$K$. If it is odd, we eliminate the field and move to the next. If we get a curve $E$ that has even rank over $\Q(\sqrt{d_1}, \sqrt{d_2})$, then we check the parity of the curves  $E^{(d_1)}, E^{(d_2)}, E^{(d_1\cdot d_2)}$ and eliminate the field if any of the curves have odd rank. If all the fields are eliminated, we move to the next elliptic curve.

After some searching we find the elliptic curve
$$ y^2 = x^3+12974641/13176900x^2 + 16/14641x$$
over the field $\Q(\sqrt{-330}, \sqrt{-671})$. One can easily compute that the ranks of $E$ and $E^{(-671)}$ are $0$ by $2$-descent in \textsf{mwrank}, but proving that the other two twists, $E^{(-330)}$ and $E^{(1871)}$, have rank 0 can not be done by $2$-descent. Actually, both curves have the $2$-primary part of their Tate--Shafarevich group isomorphic to $\Z/ 4\Z \oplus \Z/ 4\Z$, which can be checked by doing an $8$-descent in Magma. In fact, the Tate--Shafarevich group of $E^{(-330)}$ is conjecturally isomorphic to $\Z/ 12\Z \oplus \Z/ 12\Z$. However, one can prove that the twisted curves have rank 0 by approximating the $L$-value $L(E,1)$ and using Kolyvagin's result that an elliptic curve $E$ over~$\Q$ with $L(E,1)\ne0$ has rank~0. We used Sage for this computation.
\end{proof}

From Theorems \ref{t1}, \ref{t2}, and \ref{t3} we obtain the following result about the number of points of an elliptic curve over a number fields of degree 2, 3, and 4.

\begin{corollary}
\begin{enumerate}
\item An elliptic curve over a quadratic field has $1$, \dots, $16$, $18$, $20$, $24$, or infinitely many points. All of the cases occur.
\item There exist elliptic curves over cubic fields with $1$, \dots, $15$, $16$, $18$, $20$, $24$, $28$, and infinitely many points.
\item There exist elliptic curves over quartic fields with $1$, \dots, $17$, $18$, $20$, $21$, $22$, $24$, $25$, $27$, $28$, $32$, $36$, and infinitely many points.
\item Up to isomorphism, there exist only finitely many elliptic curves over cubic and quartic fields for which the number of points is not in these lists.
\end{enumerate}
\end{corollary}

The last part depends on Merel's theorem~\cite{Merel}: for all $d\ge1$, there are only finitely many groups that occur as torsion subgroups of elliptic curves over number fields of degree~$d$.  We also note that no curves with a different number of points are currently known.  Furthermore, we do not know whether all the listed possibilities occur infinitely often.

\section{Curves with prescribed torsion and positive rank}

As mentioned in the introduction, Mazur and Rubin \cite[Theorem~1.1]{mr} proved that over each number field there exists an elliptic curve with rank~0. We can reinterpret this theorem in the following way: if we look among all elliptic curves over a number field $K$, we will find a rank 0 curve. It is natural to ask whether this statement holds true if one only looks among elliptic curves satisfying some condition. We prove that the statement is not true if one looks only at elliptic curves with prescribed torsion over some fixed number field.

\begin{tm}
\label{t5}
There is exactly one elliptic curve up to isomorphism over the quartic field $\Q(i,\sqrt{5})$ having torsion subgroup $\Z /15\Z$, and it has positive rank.
\end{tm}
\begin{proof}
(Compare \cite[Theorem 5]{kn}.) Over $\Q(\sqrt{5})$ there is exactly one curve with torsion $\Z/ 15 \Z$, namely
$$
E\colon y^2 = x^3 + (281880\sqrt{5} - 630315)x + 328392630-146861640\sqrt{5},
$$
and one of the points of order~15 is
$$
P=(315-132\sqrt 5,5400-2376\sqrt 5).
$$
We take an explicit affine model of $\X_1(15)$,
$$\X_1(15)\colon y^2+xy+y=x^3+x^2,$$
which can be found in \cite{yan}, and compute
$$
\X_1(15)(\Q(\sqrt{5}))\simeq\Z/8\Z.
$$
Four of the points are cusps. The other four correspond to the pairs $(E, \pm P)$, $(E, \pm 2P)$, $(E, \pm 4P)$, $(E, \pm 7P)$. Over $\Q(\sqrt{5},i)$, no extra points of $\X_1(15)$ appear, so $E$ remains the only curve with torsion~$\Z/15\Z$. In addition, $E$ acquires the point $(-675+300\sqrt{5},(2052\sqrt{5}-4590)i)$ of infinite order over~$\Q(i,\sqrt{5})$.
\end{proof}

Finding elliptic curves with high rank and prescribed torsion has a long history, which can be seen at the webpage \cite{du}, where there is a list of over 50 references about this problem.

Finding elliptic curves over number fields with prescribed torsion and positive rank has first been done by Rabarison \cite{rab}, who studied elliptic curves over quadratic fields. He managed to find elliptic curves over quadratic fields with positive rank and torsion $\Z/11\Z$, $\Z/13\Z$, $\Z/14\Z$, $\Z/16\Z$, $\Z /3\Z \oplus \Z/3 \Z$, $\Z /3\Z \oplus \Z/6 \Z$, and $\Z /4\Z \oplus \Z/4 \Z$.

\begin{tm}
\label{t6}
There exist elliptic curves over quadratic fields whose Mordell--Weil groups contain $\Z/15 \Z\oplus\Z$, $\Z/18\Z\oplus\Z^2$, $\Z /2\Z \oplus \Z/10 \Z\oplus\Z^4$, and $\Z /2\Z \oplus \Z/12 \Z\oplus\Z^4$.
\end{tm}

\begin{proof}
We give explicit examples of such curves in Table~\ref{table:rank+quadratic}.

Suppose $E$ is an elliptic curve with exactly one point of order 2 over some number field~$K$. Note that the $2$-division polynomial of $E$, $\psi_2$, is of degree~3, so $\psi_2$ factors over $K$ as a linear polynomial times a quadratic polynomial. This implies that there is exactly one quadratic extension of $K$ over which $E$ has full $2$-torsion, which can easily be found (the splitting field of the quadratic factor).

We start with elliptic curves with high rank over $\Q$ and torsion $\Z/10\Z$ and $\Z/12 \Z$ over~$\Q$. Explicit examples of such curves are given in~\cite{du}. We take such a curve and use the fact that it has exactly one rational 2-torsion point. We then apply the above method to construct elliptic curves with high rank and torsion $\Z /2\Z \oplus \Z/10 \Z$ and $\Z /2\Z \oplus \Z/12 \Z$ over some quadratic field.
In this way, we find the two corresponding curves in Table~\ref{table:rank+quadratic}.



To find an elliptic curve with torsion $\Z/15 \Z$ and positive rank, we search for points on $\X_1(15)$ and then from them construct elliptic curves. Among those elliptic curves, we sieve for the ones that have conjecturally odd rank. Then we are left with the problem of finding a point of infinite order on such a curve. We managed to find such a point by looking at small multiples of divisors of the discriminant of the curve.


For elliptic curves over quadratic fields with torsion $\Z/18\Z$, the rank is always even, as will be shown in Theorem~\ref{theorem:false-CM}.  For curves constructed from points of~$\X_1(18)$ over quadratic fields, therefore, instead of computing the root number to find candidates for curves with positive rank, we approximate the value of the $L$-function $L(E,s)$ at $s=1$; if this vanishes, then $E$ should have rank at least~2 by the Birch--Swinnerton-Dyer conjecture.  In this way we obtain the elliptic curve in Table~\ref{table:rank+quadratic}.  The first point of infinite order was found by T. Preu using the \textsf{PseudoMordellWeilGroup()} function in Magma after a change of coordinates to simplify the Weierstrass equation.  The second point of infinite order can be found using the action of $\Z[\sqrt{-2}]$ on~$E(K)$ described in Subsection~\ref{subsec:18}.
\end{proof}

We study the same problem over cubic and quartic fields, and find many new examples of Mordell--Weil groups of elliptic curves. We will not give examples for the torsion groups $T$ such that it is trivial to find an elliptic curve with torsion~$T$ and positive rank.  Let us explain what we mean by ``trivial'' for cubic and for quartic fields.

Over cubic fields, the trivial $T$ are those that already appear as torsion groups of elliptic curves over~$\Q$, as well as $\Z/ 14 \Z$ and $\Z/ 18 \Z$.
Namely, let $E$ be an elliptic curve over~$\Q$ with positive rank and torsion $\Z/7\Z$ or $\Z/9\Z$. Then over the cubic field generated by a root of the $2$-division polynomial, $E$ has torsion $\Z/14\Z$ or $\Z/18\Z$, respectively. In this way, one can construct elliptic curves with Mordell--Weil groups $\Z/ 14 \Z\oplus \Z^5$ and $\Z/ 18 \Z \oplus \Z^4$; see~\cite{du}.

Over quartic fields, the trivial $T$ are those that already occur over the rational numbers or a quadratic field, those of the form $\Z/ 4n \Z$, where $\Z/ 2n \Z$ occurs over~$\Q$, and those of the form $\Z/2\Z \oplus T'$, where $T'$ occurs over a quadratic field and has exactly one element of order~2.
Let $E$ be an elliptic curve over~$\Q$ with positive rank and torsion $\Z/ 2n \Z$, and let $R\in E(\Qbar)$ be a point such that $2R$ generates this torsion group. Then $E$ has torsion at least~$\Z/ 4n \Z$ over the field of definition of $R$. In this way, one can construct elliptic curves with Mordell--Weil group containing $\Z/ 16 \Z\oplus \Z^6$, $\Z/ 20 \Z\oplus \Z^4$, and $\Z/ 24 \Z\oplus \Z^4$.
Now let $E$ be an elliptic curve over a quadratic field~$K$ with positive rank and torsion $T'$, with $T'$ as above. Then $E$ has torsion at least $\Z/2\Z \oplus T'$ over the quartic field~$K(E[2])$. In this way, one can construct elliptic curves over quartic fields with Mordell--Weil groups containing $\Z/ 2 \Z\oplus \Z/ 14 \Z\oplus\Z$, $\Z/ 2 \Z\oplus \Z/ 16 \Z\oplus\Z$, $\Z/2\Z \oplus \Z/18\Z \oplus \Z^2$, and $\Z/ 6 \Z\oplus \Z/ 6 \Z\oplus\Z^6$. (See \cite{mjb} for an example of an elliptic curve with Mordell--Weil group containing $\Z/ 3 \Z\oplus \Z/ 6 \Z\oplus\Z^6$ over $\Q(\sqrt{-3})$.)

\begin{tm}
\label{t7}
There exist elliptic curves over cubic fields whose Mordell--Weil groups contain $\Z/11\Z \oplus \Z^2$, $\Z/13\Z \oplus \Z$, $\Z/15\Z \oplus \Z$,  $\Z/16\Z \oplus \Z$, $\Z/20\Z \oplus \Z$, $\Z/2\Z \oplus \Z/10\Z \oplus \Z$, $\Z/2\Z \oplus \Z/12\Z \oplus \Z$ and $\Z/2\Z \oplus \Z/14\Z \oplus \Z$.
\end{tm}
\begin{proof}
We give explicit examples of such curves in Table~\ref{table:rank+cubic}.  For torsion group $\Z/11\Z$, we search through the curves constructed in \cite{jks}, and quickly find a rank two curve among the smallest cases.
For the other torsion groups, we construct elliptic curves with given torsion using \cite{jk3}, sieve for elliptic curve with conjecturally odd (and thus positive) rank, and then find points on the curves obtained.
\end{proof}

\begin{tm}
\label{t8}
There exist elliptic curves over quartic fields whose Mordell--Weil groups contain $\Z/ 17 \Z\oplus\Z$, $\Z/21\Z \oplus \Z$, $\Z/3\Z \oplus \Z/9\Z \oplus \Z$, $\Z/4\Z \oplus \Z/8\Z \oplus \Z$, and $\Z/5\Z \oplus \Z/5\Z \oplus \Z$.
\end{tm}
\begin{proof}
Explicit examples of such curves are given in Table~\ref{table:rank+quartic}.
We obtain our curves from \cite{jk4} and use the same strategy as in Theorem~\ref{t7} to obtain points of infinite order on them.
\end{proof}

From the results in this section, we can draw the following conclusion.  Let $d\le4$, and let $T$ be a group that occurs infinitely often as the torsion group of an elliptic curve over a number field of degree~$d$.  Then there exists an elliptic curve with positive rank and torsion $T$ over a number field of degree~$d$, except possibly for $d=4$ and $T=\Z/22\Z$.

\section{False complex multiplication}

\label{sec:false-cm}

In this section we describe a phenomenon that, for an elliptic curve $E$ over a quadratic field $L$ having torsion $\Z/13\Z$, $\Z/16\Z$, or $\Z/18\Z$, gives $E(L)$ a module structure over the ring $\Z[t]/(t^2-a)$ for some $a\in\Z$.  We say that $E$ has \emph{false complex multiplication\/} by~$\Q[t]/(t^2-a)$; a precise definition of false complex multiplication will be given in Subsection~\ref{subsec:K-curves}.  A similar phenomenon occurs for elliptic curves with torsion $\Z/22\Z$ over quartic fields.

\begin{tm}
\label{theorem:false-CM}
Let $E$ be an elliptic curve over a number field~$L$.
\begin{enumerate}
\item If\/ $[L:\Q]=2$ and $E$ has a rational point of order\/~$13$, then $L$ is real and $E$ has false complex multiplication by~$\Q(\sqrt{-1})$.
\label{tm13}
\item If\/ $[L:\Q]=2$ and $E$ has a rational point of order\/~$16$, then $E$ has false complex multiplication by~$\Q\times\Q$.
\label{tm16}
\item If\/ $[L:\Q]=2$ and $E$ has a rational point of order\/~$18$, then $L$ is real and $E$ has false complex multiplication by~$\Q(\sqrt{-2})$.
\label{tm18}
\item If\/ $[L:\Q]=4$ and $E$ has a rational point of order\/~$22$, then $L$ has a quadratic subfield $K$ such that $E$ is a $K$-curve over $L$ with false complex multiplication by $\Q(\sqrt{-2})$.
\label{tm22}
\end{enumerate}
\end{tm}

\begin{corollary}
Any elliptic curve over a quadratic number field with a point of order~$13$ or order~$18$, as well as any elliptic curve over a quartic number field with a point of order~$22$, has even rank.
\end{corollary}

\begin{tm}
\label{theorem:13-16}
Let $E$ be an elliptic curve defined over a quadratic field $L$ with a point of order $n=13$ or $16$, and let $\sigma$ be the generator of~$\Gal(L/\Q)$. Then
\begin{enumerate}
\item $E$ is $L$-isomorphic to~$E^\sigma$.

\label{13-16-part1}
\item $E$ has a quadratic twist (by an element $d$ of $O_L$) $E^{(d)}$ that is defined over~$\Q$.  For any such $d$, the curve $E^{(d)}$ has an $n$-isogeny defined over~$\Q$.
\end{enumerate}
\end{tm}

\subsection{Preliminaries}

\label{subsec:K-curves}

Let $L/K$ be a finite Galois extension of number fields, and let $E$ be an elliptic curve over~$L$.  Let $\Res_{L/K}$ denote the Weil restriction functor.  We let $B$ denote the Abelian variety
$$
B=\Res_{L/K}E
$$
of dimension~$[L:K]$ over~$K$.  It is known that the base change $B_L$ of $B$ to~$L$ is given by
$$
B_L\simeq\prod_{\sigma\in\Gal(L/K)}\lsup\sigma E.
$$
From this we get an isomorphism
$$
\End_L B_L\simeq\bigoplus_{\sigma,\tau\in\Gal(L/K)}
\Hom_L(\lsup\sigma E,\lsup\tau E).
$$
where the multiplication on the left-hand side corresponds to ``matrix multiplication'' on the right-hand side.  Taking Galois invariants, we get an isomorphism
$$
\End_K B\simeq\bigoplus_{\sigma\in\Gal(L/K)}\Hom_L(\lsup\sigma E,E),
$$
where the multiplication on the right-hand side is the bilinear extension of the maps
\begin{align*}
\Hom(\lsup\sigma E,E)\times\Hom(\lsup\tau E,E)&\longrightarrow
\Hom(\lsup{\sigma\tau} E,E)\\
(\mu,\nu)&\longmapsto\mu\circ\lsup\sigma\nu.
\end{align*}

\begin{definition}
Let $L/K$ be a finite Galois extension of number fields.  A \emph{$K$-curve over~$L$} is an elliptic curve $E$ over~$L$ that is isogenous to its Galois conjugates $\lsup\sigma E$ for all $\sigma\in\Gal(L/K)$.
\end{definition}

Let $E$ be a $K$-curve over~$L$, and write $R_E=\Q\otimes_\Z\End_LE$; this is either $\Q$ or an imaginary quadratic field.  Then $\Hom_L(\lsup\sigma E,E)$ is a one-dimensional $R_E$-vector space for all $\sigma\in\Gal(L/K)$, and hence $\Q\otimes_\Z\End_K(\Res_{L/K}E)$ is an $R_E$-vector space of dimension~$[L:K]$.

In the sequel, we will only be interested in the case where $L$ is a quadratic extension of~$K$.  For this we introduce the following terminology.

\begin{definition}
Let $L$ be a number field, let $E$ be an elliptic curve over~$L$, and let $F$ be an \'etale $\Q$-algebra of degree~2.  We say that $E$ has \emph{false complex multiplication} by~$F$ if there exists a subfield $K\subset L$ with $[L:K]=2$ such that $\Q\otimes\End_K(\Res_{L/K}E)$ contains a $\Q$-algebra isomorphic to~$F$.
\end{definition}
\begin{remarks}
Let $E$, $L$, $F$, and $K$ be as in the above definition.
\begin{enumerate}
\item An elliptic curve $E$ over a number field $L$ has false complex multiplication if and only if $E$ has complex multiplication or there is a subfield $K\subset L$ with $[L:K]=2$ such that $E$ is a $K$-curve.
\item Note that  $F$ is either $\Q\times\Q$ or a quadratic field.  If $E$ is an elliptic curve over~$L$ with false complex multiplication by~$F$, then the $\Q$-vector space
$$
\Q\otimes_\Z E(L)\simeq\Q\otimes_\Z(\Res_{L/K}E)(K)
$$
has a natural $F$-module structure.  If $F$ is a field, this implies that the finitely generated Abelian group $E(L)$ has even rank.
\end{enumerate}
\end{remarks}

\subsection{Families of curves with false complex multiplication constructed from involutions on modular curves}

\label{subsec:K-curve-constr}

Let $n$ be a positive integer, and let $\Y_0(n)$ be the (coarse) modular curve over~$\Q$ classifying elliptic curves with a cyclic subgroup of order~$n$, i.e.\ a subgroup scheme that is locally isomorphic to $\Z/n\Z$ in the \'etale topology.  Let $Y$ be a smooth affine curve over~$\Q$, let $E$ be an elliptic curve over~$Y$, and let $G$ be a cyclic subgroup of order~$n$ in~$E$.  Then we obtain a morphism
$$
Y\to\Y_0(n).
$$
We assume that this morphism is finite.  Furthermore, we suppose given an involution
$$
\iota\colon Y\isom Y
$$
that lifts the automorphism $w_n$ of~$\Y_0(n)$.

Pulling back $E$ via~$\iota$ gives a second elliptic curve $\iota^*E$ over~$Y$.  By the assumption that $\iota$ lifts~$w_n$, we have
$$
\iota^*E\simeq E/G.
$$
Symmetrically, $\iota^*E$ is equipped with the cyclic subgroup $\iota^*G$ of order~$n$, which corresponds to the subgroup $E[n]/G$ of~$E/G$; we have
$$
(\iota^*E)/(\iota^*G)\simeq E.
$$
In view of this, we may fix an isogeny
$$
\mu\colon \iota^*E\to E
$$
with kernel~$\iota^*G$.  Since the morphism $Y\to\Y_0(n)$ is finite and therefore dominant, we have $\Aut_Y E=\{\pm1\}$, and so $\mu$ is unique up to sign.  Pulling back $\mu$ via~$\iota$ and using the canonical isomorphism $\iota^*\iota^*E\simeq E$, we get a second isogeny
$$
\iota^*\mu\colon E\to\iota^*E
$$
with kernel~$G$.  Composing these, we get an endomorphism
\begin{equation}
\label{eqa}
a=\mu\circ\iota^*\mu\in\End_Y E.
\end{equation}
Note that this endomorphism does not depend on the choice of~$\mu$.  Its kernel is $E[n]$, so again using $\Aut_Y E=\{\pm1\}$, we conclude that
$$
a=\pm n.
$$
In particular, up to sign, $\iota^*\mu$ is the dual isogeny of~$\mu$.

Let $U$ be the complement of the scheme of fixed points of the involution~$\iota$ on~$Y$, let $U/\iota$ denote the quotient, and let $\Res_{U/(U/\iota)}$ denote the Weil restriction functor from $U$-schemes to $(U/\iota)$-schemes \cite[Section~7.6]{BLR}.  We write
$$
B=\Res_{U/(U/\iota)}E.
$$
Because the quotient map $U\to U/\iota$ is \'etale of degree~2, this is an Abelian scheme of relative dimension~2 over~$U/\iota$.  As in Subsection~\ref{subsec:K-curves}, we have
$$
\End_{U/\iota}B\simeq\End_U L\oplus\Hom_U(\iota^* E,E).
$$
We get an injective homomorphism
$$
\Z[t]/(t^2-a)\rightarrowtail\End_{U/\iota}B
$$
mapping $t$ to the endomorphism of~$B$ corresponding to $\mu\in\Hom_U(\iota^*E,E)$ under the above isomorphism.

Now let $K$ be a number field.  Specializing to arbitrary $K$-points of~$U/\iota$ gives a construction of $K$-curves, as follows.  Let $u\in(U/\iota)(K)$.  Suppose that the inverse image $v$ of~$u$ in~$U$ is irreducible, so it is of the form $\Spec L$ with $L$ a quadratic extension of~$L$.  Then the fiber $E_v$ is a $K$-curve over~$L$ and has false complex multiplication by $\Q[t]/(t^2-a)$.  In particular, if $a$ is not a square, then $E_v(L)$ has even rank.

\subsection{The modular curves $\Y_1(n)$}
Let $n\ge6$ be an integer.  The affine modular curve $\Y_1(n)$ classifying elliptic curves with a point of order~$n$ can be described as
$$
\Y_1(n)\simeq\Spec\bigl(\Z[s,t,1/n,1/\Delta]/(\phi_n)\bigr),
$$
where
$$
\Delta=-s^4t^3(t-1)^5\bigl(s(s+4)^2t^2 - s(2s^2+5s+20)t + (s-1)^3\bigr)
$$
and where
$$
\phi_n\in\Z[s,t]
$$
is an irreducible polynomial depending on~$n$.  The universal elliptic curve over $\Y_1(n)$ is given by the Weierstrass equation
$$
E\colon y^2+(1+(t-1)s)xy+t(t-1)sy=x^3+t(t-1)sx^2
$$
with the distinguished point~$(0,0)$, and $\phi_n$ is such that its vanishing is equivalent to the condition that $(0,0)$ is of order~$n$.

Since the polynomial $\phi_n$ is rather complicated for all but the smallest values of~$n$, it pays to introduce a change of variables giving a simpler-looking equation for~$\Y_1(n)$.  The new variables are called $u$ and~$v$ in the examples below.

\subsection{Torsion subgroup $\Z/13\Z$}
\label{subsec:13}

We consider the modular curve $\Y_1(13)$ over~$\Q$.  Its compactification $\X_1(13)$ has genus~2 and in particular is hyperelliptic.  There are six  rational cusps and six cusps with field of definition $\Q(\zeta_{13}+\zeta_{13}^{-1})$.

Let $(E,P)$ denote the universal pair of an elliptic curve and a point of order~13 over~$\Y_1(13)$.  The diamond automorphism
$$
\iota=\diam{5}=\diam{-5}
$$
of~$\Y_1(13)$ is an involution since $5^2$ is the identity element of $(\Z/13\Z)^\times/\{\pm1\}$.  It is a lift of the ``Atkin--Lehner involution $w_1$'' on~$\X(1)$ (i.e.\ the identity).  We note that the fixed points of~$\iota$ lie outside the cusps.

Pulling back $(E,P)$ via~$\iota$ yields a second pair $\iota^*(E,P)$, and the definition of $\iota$ gives an isomorphism
$$
\mu\colon\iota^*(E,P)\isom(E,5P)
$$
over~$\Y_1(13)$.  Pulling back $\mu$ via~$\iota$ gives another isomorphism
$$
\iota^*\mu\colon(E,P)\isom\iota^*(E,5P).
$$
We have
$$
\mu\circ\iota^*\mu\colon(E,P)\isom(E,5^2P)=(E,-P).
$$
This implies that, in the notation introduced above, we have
$$
a=-1,
$$
so the Abelian variety $B$ has the property that
$$
\End B\simeq\Z[\sqrt{-1}].
$$

In coordinates, the situation looks as follows.  We have
$$
\phi_{13} = t^3 - (s^4 + 5s^3 + 9s^2 + 4s + 2)t^2 + (s^3 + 6s^2 + 3s + 1)t + s^3.
$$
We use the change of variables
$$
\displaylines{
u=1/(s + 1) + 1/(t - 1),\quad v=u^4(t - 1) + u^2;\cr
s=\frac{v+u}{u(u+1)^2}-1,\quad t=\frac{v-u^2}{u^4}+1.}
$$
The modular curve $\Y_1(13)$ is isomorphic to the affine curve given by
\begin{align*}
v^2 - (u^3 + 2u^2 + u + 1)v + u^2(u + 1)&=0,\\
u(u+1)(u^3-u^2-4u-1)&\ne0.
\end{align*}
The six rational cusps are given by $u=0$, $u=-1$, and $u=\infty$; the six cusps defined over $\Q(\zeta_{13}+\zeta_{13}^{-1})$ are given by $u^3-u^2-4u-1=0$.

The hyperelliptic involution sends $(u,v)$ to $(u,u^3+2u^2+u+1-v)$.  A computation using the moduli interpretation shows that the hyperelliptic involution coincides with $\diam{5}=\diam{-5}$.

The specialization construction explained in Subsection~\ref{subsec:K-curve-constr} gives a family of $\Q$-curves.  Let $c\in\Q\setminus\{0,-1\}$.  The inverse image of the point defined by~$c$ under the map
$$
u\colon \Y_1(13)\to\Spec\Q[u,1/(u(u+1)(u^3-u^2-4u-1))]
$$
is the spectrum of the quadratic $\Q$-algebra
$$
L=\Q[v]/(v^2 - (c^3 + 2c^2 + c + 1)v + c^2(c + 1)).
$$
Since $\Y_1(13)$ does not have any $\Q$-rational points, $L$ is a field, and we obtain an elliptic curve over~$L$ with false complex multiplication by~$\Q(\sqrt{-1})$.

In fact, \emph{any} elliptic curve over a quadratic field with a point of order 13 comes from the above construction, as we will now show.

\begin{lemma}
\label{lemma:jac-g2}
Let $X$ be a proper, smooth, geometrically connected curve of genus~$2$ over a field~$k$, let $\iota$ be the hyperelliptic involution on~$X$, let $K$ be a canonical divisor on~$X$, and let $J$ be the Jacobian of~$X$.
\begin{enumerate}
\item An effective divisor of degree~$2$ on~$X$ is in the canonical linear equivalence class if and only it is the pull-back of a $k$-rational point of~$X/\iota$.
\item Let $S$ be a finite set of closed points of~$X$ such that every $k$-point  of~$J$ is of the form $[D'-K]$, where $D'$ is an effective divisor of degree~$2$ with support in~$S$.  Let $D$ be an effective divisor of degree~$2$ on~$X$.  Then either $D$ has support in~$S$, or $D$ lies in the canonical linear equivalence class.
\end{enumerate}
\end{lemma}

\begin{proof}
The first part is well known.  For the second part, let $D$ be an effective divisor of degree~2.  By assumption, $D$ is linearly equivalent to an effective divisor~$D'$ with support in~$S$.  If $D$ is not in the canonical linear equivalence class, then the complete linear system $|D|$ has dimension~0, so~$D=D'$.
\end{proof}

\begin{lemma}
\label{lemma:description-13}
Let $\J_1(13)$ denote the Jacobian of\/~$\X_1(13)$, and let $K$ be a canonical divisor on~$\X_1(13)$.  The group $\J_1(13)(\Q)$ is isomorphic to $\Z/19\Z$, and
every element of\/~$\J_1(13)(\Q)$ is of the form $[D-K]$ with $D$ an effective divisor of degree~$2$ supported on the cusps of\/~$\X_1(13)$.
\end{lemma}

\begin{proof}
Two distinct effective divisors of degree~2 on a curve of genus~2 are linearly equivalent if and only if they are both in the canonical linear equivalence class.  The set of effective divisors of degree~2 supported at the cusps, which has $\binom{6+2-1}{2}=21$ elements, therefore splits up into the canonical linear equivalence class consisting of 3 divisors (defined concretely by $u=0$, $u=-1$, and $u=\infty$) and 18 linear equivalence classes consisting of a single divisor.  We deduce that the effective divisors of degree~2 supported at the cusps yield 19 rational points of~$\J_1(13)$.

One can show by a 2-descent (implemented for instance in Magma that $\J_1(13)$ has trivial 2-Selmer group.  It follows that $\J_1(13)(\Q)$ is a finite group of odd order.  Together with the fact that $\J_1(13)$ has good reduction at~2, this implies that $\J_1(13)(\Q)$ injects into $\J_1(13)(\F_2)$.  One computes the zeta function of\/~$\X_1(13)_{\F_2}$ as
$$
Z(\X_1(13)_{\F_2},t)=\frac{1+3t+5t^2+6t^3+4t^4}{(1-t)(1-2t)}.
$$
Setting $t=1$ in the numerator, we see that $\J_1(13)$ has 19 points over~$\F_2$.  This proves that $\J_1(13)$ has no other rational points than the 19 found above.
\end{proof}

\begin{proof}[Proof of Theorem~\ref{theorem:false-CM}(\ref{tm13})]
Let $(E,P)$ be a pair consisting of an elliptic curve and a point of order~13 over a quadratic field~$L$.  Let $y$ be the closed point of~$\Y_1(13)$ defined by~$(E,P)$.  Since $\Y_1(13)$ has no $\Q$-rational points, the residue field of~$y$ is a quadratic extension of~$\Q$, so $y$ defines a divisor of degree~2 on~$\Y_1(13)$.  Combining Lemma~\ref{lemma:description-13} and Lemma~\ref{lemma:jac-g2}, with $S$ equal to the set of cusps of~$\X_1(13)$, we see that $y$ lies in the canonical linear equivalence class.  This implies that there is a $\Q$-rational point $x\in\Y_1(13)/\iota$ such that $y$ is the inverse image of~$x$ under the quotient map $\Y_1(13)\to\Y_1(13)/\iota$.  Therefore $(E,P)$ arises from the specialization construction described above, and $E$ has false complex multiplication by~$\Q(\sqrt{-1})$.

The modular curve $\Y_1(13)$ can be rewritten in the form
$$v^2=f(u)=u^6-2u^5+u^4-2u^3+6u^2-4u+1,$$
$$u(u-1)(u^3-4u^2+u+1)\neq 0.$$
The description of the points on $\Y_1(13)$ implies that for $y=(u,v)\in \Y_1(13)(L)$, $u$ is $\Q$-rational. As $f(u)>0$ for all $u \in \R$, we conclude that $v$ is a square root of a positive rational number, and hence $v$ is defined over a real quadratic field.
\end{proof}

\subsection{The modular curve $\Y_1(16)$ and quadratic $\Q$-curves with odd rank}
\label{subsec:16}

As noted in Subsection \ref{subsec:K-curve-constr}, a $\Q$-curve has even rank if the value $a$ defined in (\ref{eqa}) is not a square. In this subsection we show that for every elliptic curve defined over a quadratic field with torsion $\Z/16 \Z$, the value $a$ is equal to~1.  As opposed to the $\Z/13 \Z$ or $\Z/18 \Z$ cases, there do exist elliptic curves with torsion $\Z/16 \Z$ and odd rank.

Let $(E,P)$ denote the universal pair consisting of an elliptic curve and a point of order 16 over~$\Y_1(16)$. The diamond automorphism
$$
\iota=\diam{7}=\diam{-7}
$$
of $\Y_1(16)$ is an involution since $7^2$ is the identity element in $(\Z / 16 \Z)^\times$. As in Subsection \ref{subsec:13}, $\iota$ is a lift of $w_1$ on~$\X(1)$.

Pulling back $(E,P)$ via $\iota$ yields a second pair $\iota^*(E,P)$, and the definition of $\iota$ gives an isomorphism
$$
\mu\colon\iota^*(E,P)\isom(E,7P)
$$
over~$\Y_1(16)$.  Pulling back $\mu$ via~$\iota$ gives another isomorphism
$$
\iota^*\mu\colon(E,P)\isom\iota^*(E,7P).
$$
We have
$$
\mu\circ\iota^*\mu\colon(E,P)\isom(E,7^2P)=(E,P).
$$
This implies that we have $a=1$, so the Abelian variety $B$ has the property that
$$
\End B\simeq\Z\times\Z.
$$
Using this construction, we can in fact obtain $\Q$-curves with odd rank: the elliptic curve
$$E\colon y^2+(121+39\sqrt{10})xy-(3510+1107\sqrt{10})y =x^3-(3510+1107\sqrt{10})x^2,$$
taken from \cite[Th\'eor\`eme 10]{rab} has Mordell--Weil group $$E(\Q(\sqrt{10}))\simeq \Z/16\Z \oplus \Z.$$
It is also a $\Q$-curve, isomorphic to~$E^\sigma$.

As in the previous subsections, we prove that all elliptic curves with torsion $\Z/ 16 \Z$ are $\Q$-curves. In fact, as with curves having torsion $\Z/ 13 \Z$, they will be isomorphic, not just isogenous, to their Galois conjugates.

\begin{lemma}
\label{lemma:description-16}
Let $\J_1(16)$ denote the Jacobian of\/~$\X_1(16)$, and let $K$ be a canonical divisor on~$\X_1(18)$.  The group $\J_1(16)(\Q)$ is isomorphic to $\Z/2\Z\oplus\Z/10\Z$, and every element of~$\J_1(16)(\Q)$ is of the form $[D-K]$ with $D$ an effective divisor of degree~$2$ supported on the cusps of~$\X_1(16)$.
\end{lemma}

\begin{proof}
This is again similar to Lemma~\ref{lemma:description-13}.  The set of effective divisors of degree~2 supported at the cusps has $\binom{6+2-1}{2}+2=23$ elements and splits up into the canonical linear equivalence class, consisting of 4 divisors (with support in the cusps corresponding to 8-gons and 16-gons), and 19 linear equivalence classes consisting of a single divisor.  We deduce that the effective divisors of degree~2 supported at the cusps yield 20 rational points of~$\J_1(16)$.  Moreover, there are at least three such divisors $D$ (with support in the cusps corresponding to 2-, 4-, and 8-gons) such that $[D-K]$ has order~2.

By 2-descent, we find that the 2-Selmer group of~$\J_1(16)$ has dimension~2 over~$\F_2$.  Since we have three distinct points of order~2, we conclude that $\J_1(16)(\Q)$ is finite.  As in the proof of Lemma~\ref{lemma:description-18}, we use the fact that reduction modulo~3 is injective on torsion.  The zeta function of $\X_1(16)$ over~$\F_3$ is
$$
Z(\X_1(16)_{\F_3},t)=\frac{1+2t+2t^2+6t^3+9t^4}{(1-t)(1-3t)}.
$$
Setting $t=1$ in the numerator, we see that $\J_1(16)$ has 20 points over~$\F_3$.  This proves that $\J_1(16)$ has no other rational points than the 20 found above.  The only possible group structure is $\Z/2\Z\oplus\Z/10\Z$ in view of the three 2-torsion points.
\end{proof}

\begin{proof}[Proof of Theorem~\ref{theorem:false-CM}(\ref{tm16})]
This is proved in the same way as Theorem~\ref{theorem:false-CM}(\ref{tm13}), using Lemma~\ref{lemma:description-16} instead of Lemma~\ref{lemma:description-13}.
\end{proof}

\begin{proof}[Proof of Theorem~\ref{theorem:13-16}]
Part (\ref{13-16-part1}) follows from Subsections \ref{subsec:13} and~\ref{subsec:16}.
Since $E$ and~$E^\sigma$ are isomorphic, $E$ has to have a rational $j$-invariant, meaning that a twist $E^{(d)}$ of can be defined over~$\Q$.
As $E/L$ has an $n$-isogeny, so does~$E^{(d)}/L$. Let $C$ be an $n$-cycle on~$E$.
As $E$ and $E^\sigma$ are isomorphic and $\mu$ sends $C^\sigma$ to~$C$, one can see that $(E, C)$ and $(E^\sigma, C^\sigma)$ represent the same point on~$\Y_0(n)(L)$.
As $\Y_0$ is a coarse moduli space, the same point on $\Y_0(n)(L)$ is also represented by  $(E^{(d)}, C')$, where $C'$ is the corresponding $n$-cycle on $E^{(d)}$ over $L$ and $((E^\sigma)^{(d)}, (C')^\sigma)$.
Thus, one can see that $C'$ is $\sigma$-invariant, and hence $E^{(d)}$ has a $\Gal(\bar \Q/\Q)$-invariant cyclic subgroup of order $n$ (and hence a rational $n$-isogeny).
\end{proof}

\subsection{Torsion subgroup $\Z/18\Z$}
\label{subsec:18}

We consider the modular curve $\Y_1(18)$ over~$\Q$.  Its compactification $\X_1(18)$ has genus~2 and in particular is hyperelliptic.  There are six rational cusps, four cusps with field of definition $\Q(\zeta_3)$, and six cusps with field of definition $\Q(\zeta_9+\zeta_9^{-1})$.

We view $\Y_1(18)$ as classifying triples $(E,P_2,P_9)$ with $E$ an an elliptic curve, $P_2$ a point of order~2, and $P_9$ a point of order~9.  We define an involution~$\iota$ of~$\Y_1(18)$ by
$$
\iota(E,P_2,P_9)=(E/\langle P_2\rangle,Q_2,2P_9\bmod\langle P_2\rangle),
$$
where $Q_2$ is the generator of the isogeny dual to the quotient map $E\to E/\langle P_2\rangle$.

We denote the universal triple over~$\Y_1(18)$ by $(E,P_2,P_9)$.  The definition of~$\iota$ gives an isogeny of degree 2:
\begin{align*}
\mu\colon\iota^*(E,P_2,P_9)&\isom
(E/\langle P_2\rangle,Q_2,2P_9\bmod\langle P_2\rangle)\\
&\longrightarrow(E/E[2],0,2P_9\bmod E[2])\\
&\isom(E,0,4P_9).
\end{align*}
Pulling back $\mu$ via~$\iota$ gives a second isomorphism
$$
\iota^*\mu\colon(E,P_2,P_9)\longrightarrow
\iota^*(E,0,4P_9).
$$
We have
$$
\mu\circ\iota^*\mu\colon(E,P_9)\isom(E,4^2P_9)=(E,-2P_9).
$$
This implies that in the notation introduced above, we have
$$
a=-2,
$$
so the Abelian variety~$B$ has the property that
$$
\End B\simeq\Z[\sqrt{-2}].
$$

In coordinates, the situation looks as follows.  We have
\begin{align*}
\phi_{18}&=
(s^3 + 6s^2 + 9s + 1)t^4 + (s^5 + 7s^4 + 20s^3 + 19s^2 - 8s - 1)t^3 \\
&\qquad- s^2(s^2 + 11s + 28)t^2 - s^2(s^2 + 5s - 8)t - s^2(s^2 - s + 1).
\end{align*}
We use the change of variables
$$
\displaylines{
 u=-\frac{s^6 + 10 s^5 + 38 s^4 + 68 s^3 + 55 s^2 + 14 s + 1}{s(s+1)^6} t^3 \cr
 {}- \frac{s^8 + 11 s^7 + 53 s^6 + 135 s^5 + 176 s^4 + 88 s^3
 - 16 s^2 - 12 s - 1}{s(s+1)^6} t^2 \cr
 {}+ \frac{s^6 + 13 s^5 + 63 s^4 + 132 s^3 + 116 s^2 + 26 s}{(s+1)^6} t
 + \frac{2 s^5 + 3 s^4 - 8 s^3 - 17 s^2 - 5 s}{(s+1)^6}
,\cr v=u-s;\qquad
s=u-v,\quad t=\frac{(u^2 - 1)v-u^5 - 2 u^4 - 2 u^3 + u^2 + u}
{u^3 + 3 u^2 - 1}.}
$$
The modular curve $\Y_1(18)$ is isomorphic to the affine curve given by
\begin{align*}
v^2-(u^3+2u^2+3u+1)v+u(u+1)^2&=0,\\
u(u+1)(u^2+u+1)(u^3+3u^2-1)&\ne0.
\end{align*}
The hyperelliptic involution sends $(u,v)$ to $(u,u^3+2u^2+3u+1-v)$.  A computation using the moduli interpretation shows that the hyperelliptic involution coincides with the involution~$\iota$ defined above.

By specialization we obtain a family of elliptic curves with false complex multiplication by~$\Q(\sqrt{-2})$.

\begin{lemma}
\label{lemma:description-18}
Let $\J_1(18)$ denote the Jacobian of\/~$\X_1(18)$, and let $K$ be a canonical divisor on~$\X_1(18)$.  The group $\J_1(18)(\Q)$ is isomorphic to $\Z/18\Z$, and
every element of\/~$\J_1(18)(\Q)$ is of the form $[D-K]$ with $D$ an effective divisor of degree~$2$ supported on the cusps of\/~$\X_1(18)$.
\end{lemma}

\begin{proof}
This is proved in the same way as Lemma~\ref{lemma:description-13}.  The set of effective divisors of degree~2 supported at the cusps has $\binom{6+2-1}{2}+4=25$ elements and splits up into the canonical linear equivalence class, consisting of 5 divisors, and 20 linear equivalence classes consisting of a single divisor.  We deduce that the effective divisors of degree~2 supported at the cusps yield 21 rational points of~$\J_1(18)$.

Again by 2-descent, $\J_1(18)$ has trivial 2-Selmer group, so $\J_1(18)(\Q)$ is finite.  In general, if $A$ is an Abelian variety over a number field $K$ and $\mathfrak{p}\mid p$ is a prime of good reduction with $e(\mathfrak{p}/p) < p-1$, then reduction modulo~$\mathfrak{p}$ is injective on the torsion of $A(K_\mathfrak{p})$; see for instance \cite[Appendix]{katz}.  Using this, we see that $\J_1(18)(\Q)$ injects into $\J_1(18)(\F_5)$.  One computes the zeta function of\/~$\X_1(18)_{\F_5}$ as
$$
Z(\X_1(18)_{\F_5},t)=\frac{1-5t^2+25t^4}{(1-t)(1-5t)}.
$$
Setting $t=1$ in the numerator, we see that $\J_1(18)$ has 21 points over~$\F_5$.  This proves that $\J_1(18)$ has no other rational points than the 21 found above.
\end{proof}

\begin{proof}[Proof of Theorem~\ref{theorem:false-CM}(\ref{tm18})]
In the same way as in Theorem~\ref{theorem:false-CM}(\ref{tm13}), one proves that $E$ has false complex multiplication by $\Q(\sqrt{-2})$, using Lemma~\ref{lemma:description-18} instead of Lemma~\ref{lemma:description-13}.
As in the proof of Theorem~\ref{theorem:false-CM}(\ref{tm13}), we conclude that any quadratic point on~$\Y_1(18)$ is the inverse image under the quotient map $\Y_1(18)\rightarrow \Y_1(18)/\iota$ of a $\Q$-rational point on $\Y_1(18)/\iota$.

The modular curve $\Y_1(18)$ can be rewritten as
$$v^2=f(u)=u^6+2u^5+5u^4+10u^3+10u^2+4u+1,$$
$$u(u+1)(u^2+u+1)(u^2-3u-1)=0.$$
We conclude that for any quadratic point $y=(u,v)\in \Y_1(18)$, $u$ has to be $\Q$-rational, and since $f(u)>0$ for all $u \in \R$, this implies that $v$ is a square root of a positive rational number and hence an element of a real quadratic field.
\end{proof}

\subsection{Torsion subgroup $\Z/22\Z$}
In this subsection, we will prove Theorem~\ref{theorem:false-CM}(\ref{tm22}).
Throughout this subsection we will denote the modular curve $\X_1(22)$ by~$C$ and its Jacobian by~$J$.  The genus of $C$ is equal to $6$; this fact will play a crucial role.

Unless stated otherwise, curves are assumed to be complete, smooth, and geometrically integral.

To prove Theorem \ref{theorem:false-CM}(\ref{tm22}), we will characterize the points on $\Y_1(22)$ that are defined over quartic number fields.
One way of finding quartic points on $\Y_1(22)$ is by choosing a degree $4$ morphism $C\to\PP^1$.
The fibers of rational points then consist of points defined over number fields of degree at most~$4$.
Amongst other things, we will prove the following proposition.
\begin{proposition}
\label{prop:quarticfiber}
Each non-cuspidal quartic point of $C$ lies in a fiber of a rational point for some morphism
$C\to\PP^1$ of degree~$4$.
Furthermore, each non-cuspidal point of $C$ that lies in such a fiber has quartic field of definition.
\end{proposition}
Along the way we will characterize all degree $4$ morphisms $C\to\PP^1$; the $K$-curve property in the theorem will then follow from this characterization.

The curve $C$ parametrizes triples $(E, P_2, P_{11})$ with $E$ a generalized elliptic curve, $P_2$ a point of order~$2$, and $P_{11}$ a point of order~$11$.
Let us mention that $C$ has 20 cusps: 10 cusps defined over~$\Q$, whose moduli correspond to N\'eron $11$-gons and $22$-gons,
and 10 cusps defined over the quintic field $\Q(\zeta_{11} + \zeta_{11}^{-1})$, whose moduli correspond to N\'eron $1$-gons and $2$-gons.

Let $\iota$ be the following involution on $C$:
\[
\iota\colon (E, P_2, P_{11})\mapsto
\left(E/\langle P_2\rangle, Q_2,\, 4P_{11}\bmod \langle P_2\rangle\right),
\]
where $Q_2$ is the generator of the isogeny dual to $E\to E/\langle P_2\rangle$.
\begin{lemma}\label{Cmodiota}
The quotient $C/\langle\iota\rangle$ is isomorphic to an elliptic curve with $5$ rational points.
\end{lemma}
\begin{proof}
Put $\Lambda=\Z + \sqrt{-2}\Z$, $E = \C/\Lambda$, and $P_2=\sqrt{-2}/2\bmod\Lambda$.
If $P_{11}$ is $1/11 \pm 4/11\sqrt{-2}\bmod\Lambda$ or any multiple thereof, then $(E, P_2, P_{11})$ is a fixed point of~$\iota$.
This gives $10$ fixed points; the Riemann-Hurwitz formula now implies $g(C/\langle\iota\rangle)\leq 1$.
If there were more fixed points, then $g(C/\langle\iota\rangle)$ would be $0$ and thus $C$ would be hyperelliptic, which contradicts the fact that there are no elliptic curves over quadratic fields with a $22$-torsion point.

The $10$ rational cusps of $C$ map down to $5$ rational points on~$C/\langle\iota\rangle$.
The modular curve $C$ has level~$22$, thus the quotient $C/\langle\iota\rangle$ is an elliptic curve of conductor dividing~$22$.
According to \cite[Table~1]{mwr} there are $3$ such elliptic curves; they all have conductor $11$ and at most $5$ rational points.
\end{proof}
\begin{remark}
The elliptic curve in question is in fact isomorphic to $\X_1(11)$, but we will not need this in the sequel.
\end{remark}

A curve $X$ that has a degree $2$ morphism to an elliptic curve is called a \emph{bi-elliptic} curve.
An involution on $X$ that gives such a morphism by dividing it out, is called a bi-elliptic involution.
The so-called Castelnuovo--Severi inequality is useful in the study of bi-elliptic curves.
\begin{proposition}[Castelnuovo--Severi inequality, {\cite[Theorem~III.10.3]{sti}}]
\label{CasSevIneq}
Let $k$ be a perfect field, and let $X$, $Y$, and $Z$ be curves over~$k$.
Let non-constant morphisms  $\pi_Y\colon X\to Y$ and $\pi_Z\colon X\to Z$ be given, and let their degrees be $m$ and~$n$, respectively.
Assume that there is no morphism $X\to X'$ of degree $>1$ through which both $\pi_Y$ and $\pi_Z$ factor.
Then the following inequality holds:
\[
g(X) \leq m\!\cdot\! g(Y) + n\!\cdot\! g(Z) + (m-1)(n-1).
\]
\end{proposition}
\begin{corollary}\label{cor:gen6biell}
Let $k$ be a perfect field, and let $X$ be a bi-elliptic curve over $k$ of genus at least~$6$.
Then $X$ has a unique bi-elliptic involution~$\iota$.  Furthermore, there are no non-constant morphisms $X\to\PP^1$ of degree less than $4$, and every degree $4$ morphism $X\to\PP^1$ factors through $X\to X/\langle\iota\rangle$.
\end{corollary}
\begin{proof}
If there were two bi-elliptic involutions, $\iota$ and $\iota'$ say, then
$\pi_Y\colon X\to X/\langle\iota\rangle$ and $\pi_Z\colon X\to X/\langle\iota'\rangle$ would contradict Proposition~\ref{CasSevIneq}.
For the other two assertions, apply Proposition~\ref{CasSevIneq} with $\pi_Y\colon X\to X/\langle\iota\rangle$ and $\pi_Z$ any non-constant morphism $X\to\PP^1$ of degree at most~$4$.
\end{proof}
So we see that the degree~$4$ morphisms $C\to\PP^1$ are in bijection with the degree~$2$ morphisms $C/\langle\iota\rangle\to\PP^1$.

If we identify morphisms to $\PP^1$ whenever they differ by an automorphism of $\PP^1$,
then a degree $2$ morphism $C/\langle\iota\rangle\to\PP^1$ is given by a base-point-free linear system of
divisors of degree $2$ and dimension $1$ on~$C/\langle\iota\rangle$.
Since $C/\langle\iota\rangle$ is an elliptic curve, any complete linear system of divisors of degree $2$ is base-point-free and of dimension~$1$.
These are in turn in bijection with the set $\operatorname{Pic}^2(C)$ of linear equivalence classes of degree $2$ divisors on~$C/\langle\iota\rangle$.
An elliptic curve is its own Jacobian, so for any degree~$d$ the set $\operatorname{Pic}^d(C)$ is in bijection with $C/\langle\iota\rangle(\Q)$, which consists of $5$ points.

A quartic point on $C$ defines a rational point on $\Sym^4C$, and a morphism $C\to\PP^1$ defines a closed immersion $\PP^1\rightarrowtail\Sym^4C$.
So the $5$ degree $4$ morphisms $C\to\PP^1$ give us $5$ copies of $\PP^1$ in $\Sym^4C$ that are defined over~$\Q$.
We wish to prove that all rational points of $\Sym^4C$ outside these $\PP^1$'s are supported on the cusps of~$C$.

Fix any point of $C(\Q)$; this gives us a morphism
\[\phi\colon\Sym^4C\to J.\]
If $D$ is an effective divisor of degree $4$ on~$C$, then the fiber $\phi^{-1}(\phi(D))$ is
isomorphic to a projective space whose rational points form the complete linear system $|D|$ of effective
divisors that are linearly equivalent to~$D$.
 \begin{lemma}
Let $k$ be a perfect field, and let $X$ be a bi-elliptic curve over $k$ of genus at least~$6$.
Let $D$ be a divisor on $X$ of degree at most~$4$.  Then the dimension of the complete linear
system $|D|$ of divisors satisfies the following:
\[
\dim |D| =
\left\{\begin{array}{ll}
1 & \mbox{if $D$ is a fiber of a degree $4$ morphism $X\to\PP^1$};\cr
0 & \mbox{otherwise.}
\end{array}\right.
\]
In the former case, $|D|$ consists of all fibers of the same degree $4$ morphism $X\to\PP^1$.
\end{lemma}
\begin{proof}
Assume $|D|$ has positive dimension.  Let $Y$ be any subspace of $|D|$ of dimension~$1$,
and let $F\leq D$ be the fixed divisor of~$Y$.
Subtracting $F$ from all elements of~$Y$, we obtain a base-point-free linear system of dimension $1$ and degree $\deg(D-F)$ and thus a morphism $X\to\PP^1$ of degree~$\deg(D-F)$.
By Corollary~\ref{cor:gen6biell} we have $\deg D=4$ and~$F=0$.
Since linear systems of degree less than $4$ over any algebraic extension of $k$ have dimension~$0$, it follows that $\dim |D|$ cannot exceed~$1$.
We thus have $Y=|D|$, and the last assertion is immediate.
\end{proof}
This lemma immediately implies that the five $\PP^1$'s described above are fibers
of $\phi$ and furthermore that outside these $\PP^1$'s the rational points of $\Sym^4C$ map injectively into~$J(\Q)$.  It is thus interesting to know what $J(\Q)$ looks like.
\begin{lemma}
The Mordell--Weil rank of $J$ is zero.
\end{lemma}
\begin{proof}
Each isogeny factor of $J$ is a modular Abelian variety $A_f$, where is a newform of $S_2(\Gamma_1(N))$ with~$N\mid 22$.  We must prove that these $A_f$ all have Mordell--Weil rank~0.
There are two such $A_f$: one for the unique newform of level $11$ and one for the unique newform of level~$22$.
For $f$ of level $11$ we have $A_f= \J_1(11)$, which is an elliptic curve with $5$ rational points.
For $f$ of level~$22$, proven instances of the Birch--Swinnerton-Dyer Conjecture
\cite[Corollary~14.3]{kato} ensure us that $\operatorname{rk}A_f(\Q)=0$ if $L(f, 1)\not=0$.
Symbolic methods involving modular symbols can be used to verify $L(f, 1)\not=0$ (see for instance \cite[Section~3.10]{stein-thesis}); it turns out that this is indeed the case here.
\end{proof}
To further study the Diophantine properties of $\Sym^4C$ and $J$, we will use reduction modulo~ $3$; this will enable us to prove Proposition~\ref{prop:quarticfiber}.
In general, if $A$ is an Abelian variety over a number field $K$ and $\mathfrak{p}\mid p$ is a prime of good reduction with $e(\mathfrak{p}/p) < p-1$, then reduction modulo $\mathfrak{p}$ is injective on the torsion of $A(K_\mathfrak{p})$; see for instance \cite[Appendix]{katz}.
For us this means that $J(\Q)$ injects into~$J(\F_3)$.
\begin{proof}[Proof of Proposition \ref{prop:quarticfiber}]
We can compute the zeta function of $C_{\F_3}$, either by direct point counting over extensions of $\F_3$ or by expressing the Frobenius action on the Tate module in terms of the Hecke operator~$T_3$, and find
\[
Z(C_{\F_3}, t) = \frac{P(t)}{(1-t)(1-3t)}
\]
with
\[
P(t) = (1+t+3t^{2})^2\cdot(1+4t+3t^{2}-10t^{3}-29t^{4}-30t^{5}+27t^{6}+108t^{7}+81t^{8}).
\]
The $10$ rational cusps of $C$ reduce to distinct points of~$C(\F_3)$.
If we expand $Z(C_{\F_3}, t)$ as a power series, then the coefficient of $t^d$ is equal to the number of effective divisors of degree $d$ on~$C_{\F_3}$.  So from
\[
Z(C_{\F_3}, t) = 1+10t+55t^{2}+220t^{3}+720t^{4} + O(t^5)
\]
we can immediately read off that all points of $C(\F_3)$ are cusps.
The number of unordered $n$-tuples of cusps is $\binom{10 + n -1}{n}$, which is equal to $55$, $220$, and $715$ for $n=2, 3, 4$, respectively.
It follows that all divisors of degree $2$ and $3$ are supported on the cusps and that there are precisely $5$ divisors of degree $4$ that are not supported on the cusps.

We will now show that these $5$ points of $\Sym^4(C)(\F_3)$ are in the non-trivial fibers of $\Sym^4(C)\to J$.
This would immediately imply that all points of $\Sym^4(C)(\Q)$ outside these fibers are cuspidal, because of the injectivity of $\Sym^4(\Q)\to J(\F_3)$ outside these fibers.
To do this, we can simply count the number of non-cuspidal points in the non-trivial fibers over~$\F_3$.
Let a morphism $C\to\PP^1$ of degree $4$ be given.
The proof of Lemma~\ref{Cmodiota} implies that the $10$ cusps of~$C$ are mapped to $\PP^1$ in fibers of $4$, $4$, and $2$ points, respectively.  Hence, in each of the $5$ rational projective lines that we have in $\Sym^4(C)$, there are precisely $3$ points supported on the cusps.
Over $\F_3$ these lines have $\#\PP^1(\F_3)=4$ rational points, thus each of the $5$ lines has exactly $1$ non-cuspidal point, giving us $5$ points in total.
\end{proof}

\begin{corollary}
Each point on $C$ with quartic field of definition maps to a point of $C/\langle\iota\rangle$ that is defined over a quadratic field.\qed
\end{corollary}

\begin{proof}[Proof of Theorem \ref{theorem:false-CM}(\ref{tm22})]
Let $(E, P_2, P_{11})$ be the universal elliptic curve with points of order $2$ and~$11$ over~$\Y_1(22)$.
The construction in Subsection \ref{subsec:K-curve-constr} gives us an isogeny
$\mu\colon\iota^*(E, P_2, P_{11})\to (E, 0, 8P_{11})$ of degree~$2$.
From this we obtain an isomorphism
\[
\mu\circ\iota^*\mu\colon (E, P_{11}) \isom (E, 8^2P_{11}) = (E, -2P_{11}).
\]
Let $P$ be a point of $\Y_1(22)$ defined over a quartic number field~$L$.
From the above it follows that there is a degree $4$ morphism $C\to\PP^1$ mapping $P$ to a rational point and thus that $P$ lies above a point of $C/\langle\iota\rangle$ defined over a quadratic number field $K$ that is necessarily a subfield of $L$.
The results from Subsection \ref{subsec:K-curve-constr} now immediately imply that the elliptic curve $E$ associated with $P$ is a $K$-curve with false complex multiplication by $\Q(\sqrt{-2})$.
\end{proof}

\section{Applications to elliptic curves over finite fields}
\label{sec:aff}

Finding elliptic curves with positive rank and large torsion over number fields is not just a curiosity. As mentioned in the introduction, elliptic curves with large torsion and positive rank over the rationals have long been used for factorization, starting with Montgomery~\cite{mon}, Atkin and Morain~\cite{am}. In this section we argue that examining the torsion of an elliptic curve over number fields of small degree is beneficial in addition to examining the rational torsion.

A nice explicit example of the factorization of large numbers (Cunningham numbers in this case) using elliptic curves over number fields of small degree can be found in~\cite{bc}. The authors used elliptic curves over cyclotomic fields with torsion groups $\Z/3\Z \oplus \Z/6\Z$ and $\Z/4\Z \oplus \Z/4\Z$. Also, they tried to construct elliptic curves over cyclotomic fields with torsion $\Z/5\Z \oplus \Z/5\Z$ and $\Z/4\Z \oplus \Z/8\Z$ and positive rank (see \cite[4.4 and~4.5]{bc}), but failed. Note that one can find such curves in Theorem~\ref{t8}.

\begin{tm}
\label{th:reduction}
Let $m$ and $n$ be positive integers such that $m$ divides~$n$.  Let $E$ be an elliptic curve over~$\Q$, let $p$ be a prime number not dividing $n$ such that $E$ has good reduction at~$p$, and let $d$ be a positive integer.  Suppose there exists a number field~$K$ such that $E(K)$ contains a subgroup isomorphic to $\Z/m\Z\oplus\Z/n\Z$ and such that $K$ has a prime of residue characteristic~$p$ and inertia degree dividing~$d$.  Then $E(\F_{p^d})$ contains a subgroup isomorphic to $\Z/m\Z\oplus\Z/n\Z$.
\end{tm}

\begin{proof}
This follows from the fact that the $n$-torsion of~$E(K)$ reduces injectively modulo any prime of good reduction that does not divide $n$; see for example \cite[VII, Proposition 3.1]{sil}.
\end{proof}

We can apply Theorem~\ref{th:reduction} with a fixed number field~$K$, such as the splitting field of~$E[n]$. Then Chebotarev's density theorem gives a lower bound for the density of the set of primes~$p$ such that $E(\F_p)$ contains a subgroup isomorphic to $\Z/m\Z\oplus\Z/n\Z$.

The relevance for the elliptic curve factoring method is as follows. One looks for elliptic curves over fields of small degree having a given torsion subgroup~$G$. If $E$ is such a curve, then $E(\F_p)$ contains a subgroup isomorphic to~$G$ for a large density of primes~$p$. We say that an integer $m$ is \emph{$n$-smooth} for some fixed value of $n$ if all the prime divisors of $m$ are less than or equal to~$n$. As mentioned in the introduction, for the elliptic curve factoring method, one wants to choose elliptic curves~$E$ such that $|E(\F_p)|$ is smooth for many $p$.

The standard heuristic is that the larger the torsion subgroup $T$ of~$E(\Q)$, the greater the probability that $|E(\F_p)|$ is smooth. This is because $T$ injects into $E(\F_p)$ for all primes $p$ of good reduction that do not divide $|T|$, making $|E(\F_p)|$ divisible by $|T|$. However, this heuristic is too simplistic, as a curve with smaller $E(\Q)_\tors$ can have much larger torsion over fields of small degree, giving altogether a greater probability of $|E(\F_p)|$ to be smooth. We give an example of this phenomenon.

\begin{example}
\label{example1}
One can use \cite[Theorem~4.14]{jk4} (using $t=3$) to obtain an elliptic curve over $\Q$ with torsion $\Z/6\Z \oplus \Z/6\Z$ over the field $K=\Q(\sqrt{-3},\sqrt{217})$ and torsion $\Z/6\Z$ over~$\Q$. The curve is:
$$E_1:y^2 = x^3 - 17811145/19683x - 81827811574/14348907.$$
For example, 61, 67, and 73 are primes of good reduction that completely split in~$K$, so the complete torsion group of $E(K)$ injects into the finite fields with 61, 67, and 73 elements. One easily checks that the curve has 72 points over all the fields and that the groups are isomorphic to $\Z/6\Z \oplus \Z/12\Z$.
Now take
$$E_2:y^2 = x^3 - 25081083x + 44503996374.$$
The torsion of $E_2(\Q)$ is isomorphic to $\Z/7\Z$, implying that by standard heuristics (examining only the rational torsion), $|E_2(\F_p)|$ should be more often smooth than~$|E_1(\F_p)|$.  Note that both curves have rank 1 over~$\Q$, so the rank should not play a role.

We examine how often $|E_1(\F_p)|$ and $|E_2(\F_p)|$ are 100-smooth and 200-smooth if $p$ runs through the first 1000, 10000, and 100000 primes, excluding the first ten primes to get rid of the primes of bad reduction. For comparison, we also take the elliptic curve
$$E_3:y^2=x^3+3,$$
with trivial torsion group and rank 1.
In the following table, $p_n$ denotes the $n$-th prime number.

\begin{center}
\begin{tabular}{cccc}
\toprule
 & $ 30<p<p_{1010}$ & $30<p<p_{10010}$ & $30<p<p_{100010}$\\
\midrule
$\#$100-smooth $|E_1(\F_p)|$& 812 & 4843 & 22872\\
$\#$100-smooth $|E_2(\F_p)|$& 768 & 4302 & 20379\\
$\#$100-smooth $|E_3(\F_p)|$& 553 & 2851 & 12344\\
$\#$200-smooth $|E_1(\F_p)|$& 903 & 6216 & 35036\\
$\#$200-smooth $|E_2(\F_p)|$& 877 & 5690 & 32000\\
$\#$200-smooth $|E_3(\F_p)|$& 699 & 4134 & 21221\\
\bottomrule
\end{tabular}
\end{center}

We see that, contrary to what one would expect if examining only the rational torsion, $E_1$ is consistently more likely to be smooth than~$E_2$. Why does this happen? Examine the behavior of the torsion of $E_1(K)$ and $E_2(K)$ as $K$ varies through all quadratic fields. The torsion of $E_2(K)$ will always be $\Z/7\Z$ (see \cite[Theorem~2]{fuj}), while $E_1(\Q(\sqrt{-3}))_\tors\simeq \Z/3\Z \oplus \Z/6\Z$ and $E_1(\Q(\sqrt{217}))_\tors\simeq \Z/2\Z \oplus \Z/6\Z$. One fourth of the primes will split in $\Q(\sqrt{-3})$ and not in $\Q(\sqrt{217})$, one fourth vice versa, one fourth will split in neither field, and one fourth will split in both fields (and thus splitting completely in $\Q(\sqrt{-3},\sqrt{217}$)). This implies that $|E_1(\F_p)|$ is divisible by $6$, $12$, $18$, and $36$, each for one fourth of the primes, while all we can say for $|E_2(\F_p)|$ is that it is divisible by~7. We also see that $|E_3(\F_p)|$ is much less likely to be smooth than both $E_1$ and~$E_2$.

Note that these curves are by no means special; a similar result will be obtained if one chooses three other elliptic curves defined over $\Q$ of the same type, one with torsion $\Z/6\Z \oplus \Z/6\Z$ over a quartic field, one with torsion $\Z/7\Z$ over $\Q$, and one with trivial torsion over the rationals.
\end{example}

\section{Examples of curves with prescribed torsion}

The tables are arranged as follows. In the first column we give a pair $(m,n)$, meaning that the given elliptic curve has torsion isomorphic to $\Z/m \Z \oplus \Z / n\Z$. In Table~\ref{table:rank0quadratic}, the second column contains a squarefree integer $d$ indicating the base field $\Q(\sqrt d)$. In all other tables, the second column contains an irreducible polynomial~$f\in\Q[x]$ defining the base field, and $w$ denotes a root of~$f$. The curve is given in the third column, either by a quintuple $(a_1,a_2,a_3,a_4,a_6)$ representing a curve in long Weierstrass form or by a pair $(a,b)$ representing a curve in Tate form $y^2+axy+by=x^3+bx^2$.

\begin{table}[b]
\caption{Curves with prescribed torsion and rank 0 over quadratic fields.}
\label{table:rank0quadratic}
\begin{tabular}{lrl}
\toprule
$(m,n)$ & $d$ & Curve\\
\midrule
$(1,1)$ & $-1$ & $(0,0,0,0,6)$\\
$(1,2)$ &  $5$ & $(0,0,0,1,0)$\\
$(1,3)$ & $-2$ & $(0,0,0,0,4)$\\
$(1,4)$ & $-2$ & $(1,\frac18)$\\
$(1,5)$ & $-1$ & $(-2,-3)$\\
$(1,6)$ & $-1$ & $(-2,-12)$\\
$(1,7)$ & $-1$ & $(-1,2)$\\
$(1,8)$ & $-2$ & $(7,-6)$\\
$(1,9)$ & $-2$ & $(3,6)$\\
$(1,10)$ & $-2$ & $(-5,-24)$\\
$(1,11)$ & $2$ & $(\sqrt 2 +1, -\sqrt 2 +2)$\\
$(1,12)$ & $-3$ & $(43,-210)$\\
$(1,13)$ & $17$ & $(2\sqrt{17}-9, 18\sqrt{17}-74)$\\
$(1,14)$ & $-7$ & $(\sqrt{-7}+2, \sqrt{-7}+5)$\\
$(1,15)$ & $5$ & $(3,2)$\\
$(1,16)$ & $70$ & $(-\frac{31}{5},-\frac{18}{25})$\\
$(1,18)$ & $33$ & $(6+\sqrt{33},-5-\sqrt{33})$ \\
$(2,2)$ & $-1$ & $(0,0,0,1,0)$\\
$(2,4)$ & $-3$ & $(1,\frac{1}{18})$\\
$(2,6)$ & $-3$ & $(\frac{11}{10},\frac{9}{100})$\\
$(2,8)$ & $-3$ & $(-\frac{23}{7}, -\frac{30}{49})$\\
$(2,10)$ & $-2$ & $(-\frac{7}{2}, -\frac{9}{2})$ \\
$(2,12)$ & $6$ & $(\frac{29}{27}, \frac{50}{729})$\\
$(3,3)$ & $-3$ & $(1,-1,0,12,8)$\\
$(3,6)$ & $-3$ & $(\frac{9}{8}, \frac{7}{64})$\\
$(4,4)$ & $-1$ & $(1,\frac{15}{256})$\\
\bottomrule
\end{tabular}
\end{table}

\begin{table}
\caption{Curves with prescribed torsion and rank 0 over cubic fields.}
\label{table:rank0cubic}
\begin{tabular}{lll}
\toprule
$(m,n)$ & $f$ & Curve\\
\midrule
$(1,1)$ & $x^3+x^2+2$ & $(0,0,0,0,-3)$ \\
$(1,2)$ & $x^3+x^2+10$ & $(0,0,0,1,0)$\\
$(1,3)$ & $x^3+x^2+x-1$ & $(0,0,0,0,4)$\\
$(1,4)$ & $x^3+x^2-1$ & $(1,\frac{1}{2})$ \\
$(1,5)$ & $x^3+x+$1 & $(-2,-3)$\\
$(1,6)$ & $x^3+2$ & $(\frac{4}{3},\frac{2}{9})$\\
$(1,7)$ & $x^3+x+1$ & $(-1,-4)$\\
$(1,8)$ & $x^3+2x^2+1$ & $(-\frac{1}{2},-3)$\\
$(1,9)$ & $x^3+2x^2+1$ & $(-3,-12)$\\
$(1,10)$ & $x^3+x^2+3$ & $(-5,-24)$\\
$(1,11)$ & $x^3 - x^2 - 2$ & $(-2w^2 + 2w + 3, 2w^2 - 2w - 2)$\\
$(1,12)$ & $x^3+2$ & $(43,-210)$ \\
$(1,13)$ & $x^3 - x - 2$ & $(-w^2 - w - 1,-w^2 + w + 2)$\\
$(1,14)$ & $x^3 - x - 2$ & $(-1,-4)$ \\
$(1,15)$ & $x^3 + 2x - 1$ & $\bigl(\frac{w^2 - 2w + 3}{2}, \frac{-2w^2 - w + 1}{2}\bigr)$\\
$(1,16)$ & $x^3 - x^2 + 2x + 8$ & $(-4w-5,-7w^2-3w+10)$\\
$(1,18)$ & $x^3 + 3x - 2$ & $(-3,-12)$ \\
$(1,20)$ & $x^3 - x^2 - 2x - 2$ & $\bigl(\frac{-5w^2-w}{2},-14w^2-12w-8\bigr)$\\
$(2,2)$ & $x^3+2$ & $(0,0,0,-1,0)$ \\
$(2,4)$ & $x^3+2$ & $(1,-\frac{1}{2})$ \\
$(2,6)$ & $x^3+2$ & $(\frac{5}{2},-\frac{3}{4})$\\
$(2,8)$ & $x^3+2$ & $(\frac{17}{2},-15)$ \\
$(2,10)$ & $x^3 - x^2 - 1$ & $(-5w^2-3w-3,-5w^2-3w-4)$\\
$(2,12)$ & $x^3 - 2x - 2$ & $\bigl(\frac{12w^2+24w+17}{2}, \frac{-309w^2-546w-348}{4}\bigr)$\\
$(2,14)$ & $x^3 + 2x^2 - 9x - 2$ & $(3w^2 - 7w - 1, -55w^2 + 115w + 26)$\\
\bottomrule
\end{tabular}
\end{table}

\begin{table}
\caption{Curves with prescribed torsion and rank 0 over quartic fields.}
\label{table:rank0quartic}
\begin{tabular}{lll}
\toprule
$(m,n)$ & $f$ & \text{Curve}\\
\midrule
$(1,1)$ & $x^4+8x^2+4$ & $(0,0,0,0,6)$\\
$(1,2)$ & $x^4 + 3x^2 + 1$ & $(0,0,0,1,0)$\\
$(1,3)$ & $x^4 - 2x^2 + 4$ & $(0,0,0,0,4)$\\
$(1,4)$ & $x^4 - 3x^2 + 4$ & $(1,-2)$ \\
$(1,5)$ & $x^4 + 3x^2 + 1$ & $(-2,-3)$\\
$(1,6)$ & $x^4 + 3x^2 + 1$ & $(-2,-12)$ \\
$(1,7)$ & $x^4 - 3x^2 + 4$ & $(-1,-4)$ \\
$(1,8)$ & $x^4 + 26x^2 + 49$ & $(7,-6)$\\
$(1,9)$ & $x^4 - 7x^2 + 4$ & $(3,6)$\\
$(1,10)$ & $x^4 - x^2 + 1$ & $(-5,-24)$\\
$(1,11)$ & $x^4 + 2x^2 + 4$ & $(-w^3-1,-3w^3-8)$\\
$(1,12)$ & $x^4 - x^2 + 1$ & $(43,-210)$\\
$(1,13)$ & $x^4 - 38x^2 + 225$ & $\bigl(\frac{w^3-53w-135}{15}, \frac{3w^3 - 159w - 370}{5}\bigr)$\\
$(1,14)$ & $x^4 - 3x^2 + 4$ & $(-2w^2+5,-2w^2+8)$\\
$(1,15)$ & $x^4 - x^2 + 4$ & $(3,2)$\\
$(1,16)$ & $x^4 + 9x^2 + 9$ & $(-\frac{7}{5},-\frac{12}{25})$\\
$(1,17)$ & $x^4 - x^3 + x^2 - 5x - 4$ & $\bigl(\frac{5w^3-12w^2+9w-40}{4}, -4w^3+2w^2-10w+12\bigr)$\\
$(1,18)$ & $x^4 + 17x^2 + 64$ & $(2w^2+23,-2w^2-22)$\\
$(1,20)$ & $x^4 - 2x^3 + x^2 + 2$ & $(-5,-24)$ \\
$(1,21)$ & $x^4 - x^3 + 2x - 8$ & $\bigl(\frac{-w^3-6w^2-26w+80}{8}, \frac{-53w^3+118w^2-154w+200}{2}\bigr)$\\
$(1,22)$ & $x^4 - 2x^3 - x^2 + 2x + 8$ & $\bigl(\frac{w^3-12w^2+31w-20}{8}, 3w^3 - 13w^2 + 18w - 4\bigr)$\\
$(1,24)$ & $x^4 - 18x^2 - 15$ & $(1,\frac{8}{3})$\\
$(2,2)$ & $x^4 + 3x^2 + 1$ & $(0,0,0,1,0)$\\
$(2,4)$ & $x^4 - 3x^2 + 4$ & $(1,\frac{1}{18})$\\
$(2,6)$ & $x^4 + 2x^2 + 4$ & $(\frac{5}{2},\frac{3}{4})$\\
$(2,8)$ & $x^4 + 6x^2 + 4$ & $(-\frac{7}{3},-10)$\\
$(2,10)$ & $x^4 + 18x^2 + 25$ & $(-\frac{7}{2},-\frac{9}{2})$\\
$(2,12)$ & $x^4 + 9x^2 + 9$ & $(1,\frac{8}{3})$\\
$(2,14)$ & $x^4 - x^3 + 3x^2 + 3x + 2$ & $(w^3-2w^2+3w+3, w^3-2w^2+3w+6)$\\
$(2,16)$ & $x^4 + 2002x^2 + 116281$ & $\bigl(\frac{2329}{2695}, -\frac{366}{2401}\bigr)$\\
$(2,18)$ & $x^4-x^2-8$ & $(2w^{2}+5,-2w^{2}-4)$ \\
$(3,3)$ & $x^4 - x^2 + 4$ & $(1,-1,0,12,8)$ \\
$(3,6)$ & $x^4 - x^2 + 4$ & $(\frac{9}{8},\frac{7}{64})$\\
$(3,9)$ & $x^4 + 294x^2 + 2601$ & $\bigl(\frac{11w^2+129}{24}, \frac{269w^2+2463}{24}\bigr)$\\
$(4,4)$ & $x^4 + 1$ & $(1,\frac{1}{8})$\\
$(4,8)$ & $x^4 + 541x^2 + 72900$ & $\bigl(\frac{431}{690},-\frac{259}{529}\bigr)$\\
$(5,5)$ & $x^4 + x^3 + x^2 + x + 1$ & $(-10,-11)$\\
$(6,6)$ & $x^4 + 5x^2 + 1$ & $(\frac{9}{8},\frac{7}{64})$\\
\bottomrule
\end{tabular}
\end{table}

\begin{sidewaystable}

\vskip 13cm

\begin{center}
\caption{Curves with prescribed torsion and positive rank over quadratic fields.}
\begin{tabular}{llll}
\toprule
$(m,n)$& $f$ or $d$ & Curve& Independent points of infinite order\\
\midrule
$(1,15)$& $x^2-x-86$& $\bigl(\frac{10w+493}{448}, \frac{10w+45}{448}\bigr)$& $\bigl(\frac{-w-274}{3584},\frac{-2455w-20382}{200704}\bigr)$\\
$(1,18)$& $x^2+163x+12$& $\bigl(\frac{25105w + 2071}{216}, \frac{634768555w + 46752805}{7776}\bigr)$& $\bigl(\frac{3673w + 223}{486}, \frac{150110959w + 11056609}{8748}\bigr)$, $\bigl(\frac{-112579w-8293}{6}, \frac{-3011095399w-221775913}{288}\bigr)$\\
$(2,10)$& $55325286553$& $\bigl(-\frac{1001929453}{87419475}, -\frac{1089348928}{87419475}\bigr)$&
$(-\frac{76249664}{9062625}, -\frac{3294239461376}{55961709375})$, $(-\frac{1378903694270734}{47323818070815}, -\frac{233856747339051186962702}{6331823076742017702705})$,\\
& & & $(-\frac{317897024}{55559933}, -\frac{54763043233792}{3408435209751})$, $(-\frac{10158696384}{631362875}, -\frac{66880771114752}{779733150625})$\\
$(2,12)$& $2947271015$& $\bigl(\frac{1024873209359}{27734204981}, -\frac{543206429719981170}{2187369012646489}\bigr)$& rank $\ge4$, see \cite{du}\\
\bottomrule
\end{tabular}
\end{center}
\label{table:rank+quadratic}

\vskip.5cm

\begin{center}
\caption{Curves with prescribed torsion and positive rank over cubic fields.}
\begin{tabular}{llll}
\toprule
$(m,n)$& $f$& Curve& Independent points of infinite order\\
\midrule
$(1,11)$& $x^3-x^2-12$& $\bigl(\frac{-3w^2 + 3w - 8}{4}, -3w^2 + 3w - 12\bigr)$& $(2w^2 - 2w, 3w^2 + 15w - 30)$, $\bigl(4w^2 + 6w + 14,\frac{-29w^2 - 35w - 116}{2}\bigr)$\\
$(1,13)$& $x^3 - x^2 - 2x - 24$& $\bigl(\frac{-w^2 - 3w + 3}{3}, \frac{-5w^2 - 5w - 24}{3}\bigr)$& $\bigl(\frac{12w^2+10w+30}{9},\frac{20w^2 + 88w + 516}{27}\bigr)$\\
$(1,15)$& $x^3 + x^2 - 2x - 6$& $\bigl(\frac{7w^2 + 20w - 58}{8}, \frac{7w^2 + 20w - 66}{8}\bigr)$& $\bigl(5w^2+4w-24, \frac{21w^2 + 108w - 270}{2}\bigr)$\\
$(1,16)$& $x^3 - 5x^2 + 8$& $\bigl(\frac{-w^2 - 2w + 2}{2}, -4w^2 - w + 4\bigr)$& $\bigl(\frac{173w^2 - 62w - 296}{2}, \frac{8039w^2 - 3000w - 13896}{2}\bigr)$\\
$(1,20)$& $x^3 - 17x - 6$& $\bigl(\frac{3w^2 - 3w - 118}{80}, \frac{75w^2 + 114w - 1872}{320}\bigr)$& $\bigl(\frac{-21w^2 - 33w + 522}{64}, \frac{351w^2 + 414w - 8316}{1280}\bigr)$\\
$(2,10)$& $x^3 - x^2 - 3$& $\bigl(\frac{-5w^2 + 2w - 1}{5}, \frac{-42w^2 - 42w - 81}{25}\bigr)$& $\bigl(\frac{9w^2 - 3w + 18}{5}, \frac{30w^2 + 6w + 27}{5}\bigr)$\\
$(2,12)$& $x^3 - x^2 + x - 16$& $\bigl(1, \frac{-532w^2 + 27560w -
106048}{88209}\bigr)$& $\bigl(\frac{38w^2 - 10w + 512}{243}, \frac{1566w^2 +
2810w + 7264}{6561}\bigr)$\\
$(2,14)$ & $x^{3}-x^{2}-166x-536$ & $\bigl(\frac{7w^{2}+108w+322}{26}, \frac{-631w^{2}-8667w-22964}{169}\bigr)$ & 
$\bigl(\frac{315061w^{2}+4358637w+11743820}{3237013}, \frac{-289420914w^{2}-3975201306w-10528526328}{42081169}\bigr)$\\
\bottomrule
\end{tabular}
\end{center}
\label{table:rank+cubic}

\vskip.5cm

\begin{center}
\caption{Curves with prescribed torsion and positive rank over quartic fields.}
\begin{tabular}{llll}
\toprule
$(m,n)$& $f$& Curve& Independent points of infinite order\\
\midrule
$(1,17)$& $x^4 - 2x^3 + 2x^2 + 2x - 4$& $(-w^3 + w^2 + 2w - 1, 3w^3 - 6w^2 - 2w + 6)$& $(-w^3 + 2w^2 + 2w - 2, 2w^3 - 4w^2 - 4w + 4)$\\
$(1,21)$& $x^4 - x^3 - 2x^2 + 10x - 4$& $\bigl(\frac{-2w^3 + 7w^2 - 23w + 17}{7}, \frac{-135w^3 + 440w^2 - 715w + 246}{7}\bigr)$& $\bigl(\frac{8w^3 - 23w^2 + 35w - 4}{7},\frac{-190w^3 + 610w^2 - 984w +316}{7}\bigr)$\\
$(3,9)$& $x^4 - 2x^2 + 4$& $(w^3-5w^2+8w-3, 28w^3 - 59w^2 + 32w + 20)$& $(-w^2 + 4w - 4, 32w^3 - 32w^2 - 52w + 96)$\\
$(4,8)$& $x^4 - 3x^2 + 4$& $(\frac{17}{2},-15)$& $(\frac{5}{2},\frac{25w^2-50}{4})$\\
$(5,5)$& $x^4 + x^3 + x^2 + x + 1$& $(-\frac{88}{93},-\frac{181}{93})$& $(\frac{2}{3},\frac{7}{3})$\\
\bottomrule
\end{tabular}
\end{center}
\label{table:rank+quartic}

\end{sidewaystable}

\end{document}